\documentclass[12pt]{amsart}
\usepackage[utf8]{inputenc}
\usepackage[all]{xy}
\usepackage[IL2]{fontenc}
\usepackage[margin=1in]{geometry}

\usepackage{amsmath,amsthm,amsfonts, amssymb, amsthm, doi}
\usepackage{graphicx}
\usepackage[foot]{amsaddr} % To put the author information into the footnote on the title page.

\usepackage{lineno}
\usepackage{float}
\usepackage[inline]{enumitem}
\setenumerate[0]{label=\textbf{(\alph*)}}
\usepackage{longtable}
\usepackage{mathtools}
\usepackage[dvipsnames]{xcolor}
\usepackage{hyperref}
\hypersetup{
	colorlinks   = true,
	citecolor=blue,
	linkcolor=blue,
	urlcolor=blue,
	pdfauthor    = {Martin Ra\v{s}ka},
	pdftitle     = {Representing multiples of m in real quadratic fields as sums of squares},
	pdfkeywords  = {quadratic fields, sum of squares, indecomposable},
	pdfcreator   = {Pdflatex},
	pdfpagemode  = UseNone,
	pdfstartview = FitH
}
\newtheorem{innercustomthm}{Theorem}
\newenvironment{customthm}[1]
{\renewcommand\theinnercustomthm{#1}\innercustomthm}
{\endinnercustomthm}

\theoremstyle{plain}
\newtheorem{theorem}{Theorem}
\newtheorem*{theorem*}{Theorem}

\newtheorem{lemma}[theorem]{Lemma}
\newtheorem{cor}[theorem]{Corollary}
\newtheorem{proposition}[theorem]{Proposition}

\DeclarePairedDelimiter\floor{\lfloor}{\rfloor}

\newcommand{\Z}{\mathbb{Z}}
\newcommand{\Q}{\mathbb{Q}}
\newcommand{\kO}{\mathcal{O}}
\newcommand*{\DOI}[1]{\href{http://dx.doi.org/#1}{DOI: #1}}

\begin{document}

\author{Martin Ra\v{s}ka$^1$}

\title{Representing multiples of \texorpdfstring{$m$}{m} in real quadratic fields as sums of squares}

\address{$^1$Charles University, Faculty of Mathematics and Physics, Department of Algebra, 
Sokolovsk\'{a}~83, 18600 Praha 8, Czech Republic}
\email{\hangindent=8em\hangafter=1 raska.martin@gmail.com}


\thanks{The author was supported by student faculty grant of the Faculty of Mathematics and Physics of Charles University, by project PRIMUS/20/SCI/002 from Charles University and by SVV-2020-260589. \\
}
%ORIGINAL:
%\thanks{The author was supported by student faculty grant of the Faculty of Mathematics and Physics of Charles University, by project PRIMUS/21/SCI/014 from Charles university and by SVV-2020-260589. \\}

\maketitle

\begin{abstract}
    We study real quadratic fields $\Q(\sqrt{D})$ such that, for a given rational
integer $m$, all $m$-multiples of totally positive integers are sums of
squares. We prove quite sharp necessary and sufficient conditions for
this to happen.
Further, we give a fast algorithm that verifies this property for specific $m$, $D$ and we give complete results for $m \leq 5000$.
\end{abstract}

\section{Introduction}\label{sec:intro}

Throughout history, sums of squares were often in the focus of mathematicians. For number fields, the main question was resolved by Siegel \cite{siegel} in $1921$, when he proved Hilbert's conjecture that in every number field, every totally positive number can be represented as the sum of four square elements of the field. 

For rings of integers in number fields, we have the well-known result by Lagrange (1770) that every positive rational integer is the sum of four squares. The fact that all totally positive integers can be represented as the sum of squares was later shown to be quite exceptional. While Maa\ss{} \cite{mass} showed that it is also true for $\Q(\sqrt{5})$, Siegel \cite{Si} proved that $\Q$ and $\Q(\sqrt{5})$ are the only totally positive number fields where this can hold. 

The focus of this paper is on real quadratic fields $\Q(\sqrt{D})$ and their totally positive integers $\kO^+$. In $1973$, Peters \cite{P} proved that the \textit{Pythagoras number} of rings of integers $R$ in quadratic fields is at most $5$, meaning by definition that an element of $R$ is the sum of any number of squares in $R$ if and only if it is the sum of $5$ squares in $R$. For $D\leq 7$, the Pythagoras number is even smaller, e.g. $3$  for $D \in \{2,3,5 \}$. Summary of these result can be found in \cite[Sec. 3]{biquad}, where the authors also study a similar question in biquadratic fields. Considering the simplest cubic fields, Tinkov\'a \cite{tin} showed Pythagoras number to be often $6$.

However, as we mentioned above, in quadratic fields, all elements of $\kO^+$ are sums of squares only for $D = 5$. The natural continuation is further question when all elements of $m\kO^+$ (i.e. all $m$-multiples of totally positive integers) can be represented as the sum of squares for a fixed positive rational integer $m$. Recently, Kala and Yatsyna \cite{sosil} proved that every element of $2\kO^+$ is the sum of squares if and only if $D \in \{2,3,5 \}$. In the general case $m\kO^+$, they obtained the following theorem, which gives necessary and sufficient bounds for this to happen.

\begin{theorem}{\textup{\cite[Theorem 4]{sosil}}}\label{cor:2}
	Let $K=\mathbb{Q}(\sqrt{D})$ with $D\geq 2$ squarefree. Let $\kappa=1$ if $D\equiv 1\pmod 4$ and $\kappa=2$ if $D\equiv 2,3\pmod 4$.
\begin{enumerate}[label=\alph*)]
	\item If 
	$m<\frac{\kappa\sqrt D}4$, then \emph{not} all elements of $m\kO^+$ are represented as the  sum of squares in $\kO$.
	\item If $m\geq \frac{D}2$, then all elements of $\kappa m\kO^+$ are sums of five squares in~$\kO$.
	\item If $m$ is odd and $D\equiv 2,3\pmod 4$, then there exist elements of $m\kO^+$ that are \emph{not} sums of squares in $\kO$.
\end{enumerate}
\end{theorem}

This paper improves these results for general $m$. In Section \ref{sec:bounds}, the main goal is to further restrict the possibilities for $D$ in terms of $m$, which results in the following theorem:

\begin{theorem}\label{theorem:main}
Let $K=\Q(\sqrt{D})$ with $D\geq 2$ squarefree and a positive integer $m$. If $\sqrt{D}$ lies in one of the following intervals, then \emph{not} all elements of $m\kO^+$ can be represented as the sum of squares:
\begin{enumerate}
\item $\left [ \frac{m}{2}+4 ,\infty\right)$,
$ \left[\frac{m}{2i}+iC_1,  \frac{m}{2(i-1)}-(i-1)C_2 \right ]$ for integer $i>1$ and $D \equiv 2,3 \pmod{4}$,
\item $\left [ \frac{m}{2}+8 ,\infty\right)$,
$ \left[ \frac{m}{2i}+2iC_1,  \frac{m}{2(i-1)}-2(i-1)C_2 \right ]$ for integer $i>1$, $D \equiv 1 \pmod{4}$ and even $m$,
\item $\left [ m+4 ,\infty\right)$, $ \left[ \frac{m}{2i+1}+(4i+2)C_1,  \frac{m}{2i-1}-(4i-2)C_2 \right ]$ for integer $i>0$, $D \equiv 1 \pmod{4}$ and odd $m$,
\end{enumerate}
where constants $C_1$, $C_2$ are defined as $C_1 = \sqrt{24+16\sqrt{2}}$ and $C_2 = \sqrt{48+24\sqrt{3}}$.
\end{theorem}
It should be noted that for fixed $m$, only finitely many of these intervals are non-empty (for further details see Section \ref{sec:bounds}).
This not only improves the necessary upper bounds given by Theorem \ref{cor:2} but also gives better insight into the overall structure of the problem.

Sections \ref{sec:peters} and \ref{sec:indecom}, on the other hand, focus on proving that for specific fields, all elements of $m\kO^+$ are sums of squares in $\kO$. 
For example, Theorem \ref{theorem:m=4} gives detailed proof that all elements of $4\kO^+$ are sums of squares if and only if $D \in \{ 2,3,5,6,7,10,11,13 \}$. 

Conversely if $D$ is of the form of $t^2-1$ or $(2t+1)^2-4$, Theorems \ref{theorem:t2-1} and \ref{theorem:2t-4} provide a full characterization of all $m$ that satisfy the given condition. One could give similar results for $D$ in other quadratic families.

In Section \ref{sec:alg} we give a general algorithm that completely determines all of these fields satisfying this property for arbitrary $m$.  The algorithm uses the structure of indecomposable elements of $\kO^+$ and, for fixed $m$ and $D$, it proves the statement or finds a counterexample with time complexity $O(\sqrt{D}(\log(D))^2)$. Implementation of this algorithm in C++ as well as specific results for $m \leq 5000$ are available at\\ 
\noindent \url{https://github.com/raskama/number-theory/tree/main/quadratic}.

\section{Preliminaries}\label{sec:prelim}

Throughout the work, $D\geq 2$ will denote a squarefree rational integer.
We will work with real quadratic fields $K = \Q(\sqrt{D})$, more specifically with their rings of integers denoted by $\kO_K$ or simply $\kO$.
For quadratic fields, we know the basis of $\kO$ is $\{1, \omega_D\}$, where $\omega_D = \sqrt{D}$ if $D\equiv 2,3 \pmod{4}$ and $\omega_D = (1+\sqrt{D})/2$ if $D \equiv 1 \pmod{4}$. 

An algebraic integer $\alpha = x+y\sqrt{D} \in \kO$ is $\emph{totally positive}$ if $\alpha >0$ and $\alpha' > 0$, where $\alpha' = x-y\sqrt{D}$ is the Galois conjugate of $\alpha$. We denote by $\kO^{+}$ the set of all totally positive algebraic integers.

The totally positive integers $\kO^{+}$ can be viewed as an additive semigroup. In Section \ref{sec:indecom}, the notion of \emph{indecomposable} elements in this semigroup will be useful. They are by definition such elements $\alpha \in \kO^{+}$ that cannot be written as sum of two other elements $\alpha = \beta +\gamma$ and $\beta,\gamma \in \kO^{+}$. It is clear that indecomposable elements generate this whole semigroup.

Indecomposable elements in quadratic cases are closely tied to the continued fractions. The continued fractions of quadratic integers, their recurrence relations and structure were studied by Perron \cite{per}. The structure of indecomposable elements was later described by Dress and Scharlau \cite{Sindecomp}. In the following paragraph, we state these well-known facts, adopting the notation used by \cite{indecompos}.

Denote $\omega_D = [u_0,\overline{u_1,\ldots,u_s}]$ the continued fraction of $\omega_D$ and $p_i/q_i$ its convergents. The sequences $(p_i)$ and $(q_i)$ then satisfy the recurrence relation
\begin{align}
X_{i+2} = u_{i+2}X_{i+1}+X_i \hspace{0.5em} \text{for} \hspace{0.5em}i\geq -1 \label{reccurCF}
\end{align}
with the initial conditions $q_{-1} = 0$, $p_{-1}=q_0=1$ and $p_0 = u_0$ \cite[\S 1]{per}.
Further, denote $\alpha_i = p_i-q_i\omega'_D$ (where $\omega'_D = (1-\sqrt{D})/2$ if $D\equiv 1 \pmod{4}$ and $-\sqrt{D}$ otherwise), $\alpha_{i,r} = \alpha_i+r\alpha_{i+1}$ and let $\varepsilon >1$ be the fundamental unit of $\kO$.
 Then the following facts are true (see e.g. \cite{Sindecomp}):
\begin{itemize}
    \item The indecomposable elements in $\kO^{+} $ are exactly $\alpha_{i,r}$ with odd $i\geq -1$ and $0\leq r \leq u_{i+2} -1$, together with their conjugates.
    \item The sequence $(\alpha_i)$ satisfies the recurrence relation (\ref{reccurCF}).
    \item The equality $\alpha_{i,u_{i+2}} = \alpha_{i+2,0}$ holds.
    \item The fundamental unit satisfies $\varepsilon = \alpha_{s-1}$ and $\alpha_{i+s} = \varepsilon\alpha_i$ for all $i\geq -1$.
    \item For the smallest totally positive unit $\varepsilon^+>1$, we have $\varepsilon^+=\varepsilon$ if $s$ is even and $\varepsilon^+=\varepsilon^2 = \alpha_{2s-1}$ if $s$ is odd.
\end{itemize}

\section{Bounds on \texorpdfstring{$m$}{m} and \texorpdfstring{$D$}{D}} \label{sec:bounds}

If we look at all elements of the form $x+k\sqrt{D}$ for some fixed $k$, there exists minimal $x = 1+\floor{k\sqrt{D}}$ for which the element is still totally positive. As we are interested in sums of squares and $\sum(a_i+b_i\sqrt{D})^2 = \sum(a_i^2+b_i^2D) + \sum 2a_ib_i\sqrt{D}$, it is therefore useful to study the minimum of $\sum a_i^2+b_i^2D$ for some fixed value of $\frac{k}{2} = \sum a_ib_i$. If this minimum is larger than $1+\floor{k\sqrt{D}}$, we obtain a totally positive element which can't be represented as the sum of squares.

These ideas will be properly formulated in the proofs later. We will also see that we can impose on $a_i$, $b_i$ further restrictions that they are non-negative and sometimes congruent modulo $2$. As can be seen in the following lemmata, this minimization question can be more easily answered if we know how large is $D$ in comparison to $m$. To simplify our future formulations, it is convenient to introduce the following sets of intervals depending on $m$:
\begin{enumerate}
    \item 
$I_t(m) =
\left\{
	\begin{array}{ll}
		\left[\frac{m^2}{t(t+1)},\frac{m^2}{(t-1)t}\right]  & \mbox{if } t>1 \\
		\left[\frac{m^2}{2},\infty \right) & \mbox{if } t = 1
	\end{array}
\right.
$
    \item 
$J_t(m) =
\left\{
	\begin{array}{ll}
		\left[\frac{m^2}{t(t+2)},\frac{m^2}{(t-2)t}\right]  & \mbox{if } t>2\\
		\left[\frac{m^2}{t(t+2)},\infty \right) & \mbox{if } t \in \{1,2\}
	\end{array}
\right.
$
\end{enumerate}
Here, $m$ and $t$ are positive rational integers and we will use just $I_t$, $J_t$ if the value of $m$ is clear from the context. It should be also noted that $\bigcup_{t\geq 1} I_t(m)  = \bigcup_{t\geq 1} J_{2t}(m)= \bigcup_{t\geq 1} J_{2t-1}(m) = (0,\infty)$ for arbitrary fixed $m$. 
\begin{lemma} \label{lemma:basic}
Let $a_i$, $b_i$, $1\leq i\leq n$ be non-negative integers satisfying  $\sum_{i=1}^n a_ib_i = m$ for a fixed positive integer $m$. Let $D$ be a real number and $t$ be a positive integer such that $D\in I_t = I_t(m)$. Then
 \[\sum\limits_{i=1}^n a_i^2+b_i^2D \geq \frac{m^2}{t} +tD. \]
\end{lemma}
\begin{proof}
Let $k = \sum b_i^2$. By Cauchy-Schwarz inequality 
\[\sum a_i^2 \geq \frac{\left(\sum a_ib_i\right)^2}{\sum b_i^2} = \frac{m^2}{k}.\]

Now $\sum a_i^2+b_i^2D \geq \frac{m^2}{k} +kD$, so it is enough to prove $\frac{m^2}{k}+kD\geq \frac{m^2}{t}+tD$, which is equivalent to $(m^2-tkD)(t-k)\geq 0$. If $t=k$, this is true. Otherwise since $D\in I_t$, i.e. $t(t-1)D\leq m^2\leq t(t+1)D$, both factors $m^2-tkD$ and $t-k$ always have the same sign. Either $t\geq k+1$, and thus $m^2 \geq t(t-1)D \geq tkD$, or $t\leq k-1$ and $m^2 \leq t(t+1)D \leq tkD$. In both cases $(m^2-tkD)(t-k)\geq 0$.
\end{proof}
In other words, the function $\frac{m^2}{x}+xD$ attains its minimum over positive integers at $x = t$.

We will also need a version of the previous lemma with an extra condition on the parity of $a_i$, $b_i$. Then the inequalities can be refined in the following way:

\begin{lemma} \label{lemma:parity}
Let $a_i \equiv b_i \pmod{2}$, $1\leq i\leq n$ be non-negative integers satisfying  $\sum_{i=1}^n a_ib_i = m$ for a fixed positive integer $m$. Let $D$ be a real number and $t\equiv m \pmod{2}$ a positive integer such that $D \in J_t = J_t(m)$. Then
 \[\sum\limits_{i=1}^n a_i^2+b_i^2D \geq \frac{m^2}{t} +tD. \]
\end{lemma}
\begin{proof}
Let $k = \sum b_i^2$. The additional condition implies $m = \sum a_ib_i \equiv \sum b_i^2 = k \pmod{2}$. In the same fashion as in the proof of the previous lemma, it is sufficient to look at the minimum of the function $f(x) = \frac{m^2}{x}+xD$ over positive integers congruent to $m$ modulo $2$. Denote $t_0$ the positive integer such that $D \in I_{t_0}$. As was seen in the proof of the previous lemma, the minimum of $f(x)$ over all positive integers is equal to $f(t_0)$. If $t_0\equiv m \pmod{2}$, we have $D \in I_{t_0} \subsetneq J_{t_0}$. Since the intersection of the interiors of any two intervals $J_k$, $J_l$, such that $k\equiv l\pmod{2}$, is empty, we have $t=t_0$ and we are done. In the case $t_0\not \equiv m \pmod{2}$, the minimum of this function over positive integers congruent to $m$ modulo $2$ must be either $f(t_0-1)$ or $f(t_0+1)$ since this function has only one local minimum over positive reals. Since $D \in I_{t_0} \subsetneq J_{t_0-1} \cup J_{t_0+1}$, $t$ is equal to either $t_0-1$, or $t_0+1$. It can be easily seen that $f(t_0-1)\leq f(t_0+1)$ if and only if $D\geq \frac{m^2}{(t_0-1)(t_0+1)}$, i.e. $D \in J_{t_0+1}$. This concludes the proof.
\end{proof}

Using the Lemma \ref{lemma:basic} for $t=1$, if $D \geq \frac{m^2}{2}$, then $\sum a_i^2 +b_i^2D \geq m^2+D$. If we add the parity condition, Lemma \ref{lemma:parity} implies the minimum is $m^2+D$ for $D \geq \frac{m^2}{3}$ and odd $m$, resp.  $\frac{m^2}{2}+2D$ for $D \geq \frac{m^2}{8}$ and even $m$. Even though we won't use these special cases explicitly, they will give rise to the general necessary bounds for $D$ in terms of $m$.

Both of these lemmata can be utilized to greatly restrict possible values of $D$ for which all elements of $m\kO^+$ are represented as the sum of squares in $\kO$. For fixed $k$, one could simply consider the smallest totally positive element $\alpha$ of the form $a+k\sqrt{D}$ and look when $m\alpha$ can be represented as the sum of squares. Combining these results for all possible $k$ and $I_t$, $J_t$ gives rise to the following proposition.

\begin{proposition} \label{proposition:main}
Let $K=\Q(\sqrt{D})$ with $D\geq 2$ squarefree and a positive integer $m$. 
\begin{enumerate}
    \item If $D\equiv 2,3 \pmod{4}$, $t, k \in \Z_{>0}$, $D\geq m$ and 
    \[\sqrt{D} \in
     \left [\frac{mk}{2t}+\frac{\sqrt{m}}{\sqrt{t}}, \frac{mk}{2(t-1)}-\frac{\sqrt{m}}{\sqrt{t-1}}\right]
    ,\] 
    then \emph{not} all elements of $m\kO^{+}$ are represented as the sum of squares in $\kO$. This interval is non-empty only for 
    \[m\geq \frac{4t\left(t-1\right)\left(2t-1+2\sqrt{t(t-1)}\right)}{k^2}.\]
    \item If $D\equiv 1 \pmod{4}$, $t, k \in \Z_{>0}$, $t\equiv mk \pmod{2}$, $D\geq 4m$ and 
    \[\sqrt{D} \in
     \left
    [\frac{mk}{t}+\frac{2\sqrt{m}}{\sqrt{t}}, \frac{mk}{t-2}-\frac{2\sqrt{m}}{\sqrt{t-2}}\right]
    ,\]
    then \emph{not} all elements of $m\kO^{+}$ are represented as the sum of squares in $\kO$. This interval is non-empty only for 
    \[m\geq \frac{t\left(t-2\right)\left(2t-2+2\sqrt{t(t-2)}\right)}{k^2}.\]
\end{enumerate}
For $t=1$ if $D\equiv 2,3 \pmod{4}$, resp. $t\in \{1,2\}$ if $D\equiv 1 \pmod{4}$, we define the right bound of the intervals to be $\infty$ and they are therefore always non-empty.
\end{proposition}

\begin{proof}
\textbf{(a)} Let's start with the case $D\equiv 2,3 \pmod{4}$. 

For an arbitrary positive integer $k$, consider $\alpha = \floor*{k\sqrt{D}}+1 +k\sqrt{D}$ and suppose $m\alpha = \sum (\alpha_i)^2$ for some $\alpha_i = a_i+b_i\sqrt{D} \in \kO$. Observe that $\alpha' = \floor*{k\sqrt{D}}+1 -k\sqrt{D} <1$. Without loss of generality, we can choose $a_i\geq0$ for all $i$ and $b_i\geq 0$ if $a_i=0$. For the sake of contradiction, suppose there exists some $i$ such that $b_i < 0$ and $a_i>0$. We then have $m\alpha' \geq (\alpha_i')^2 = a_i^2+b_i^2D+2a_i(-b_i)\sqrt{D}\geq1+D+2\sqrt{D}\geq m$ for $D\geq m$, which is impossible because $\alpha'<1$. Therefore, we can assume $b_i\geq 0$ for all $i$.

If we compare the irrational parts in the expression of $m\alpha$, we get $mk = 2\sum a_ib_i$. The choice of $k=1$ immediately eliminates all odd $m$, so we can consider only even $m$. Comparing rational parts,  we can apply Lemma \ref{lemma:basic} for an arbitrary $t$ to get 
\[mk\sqrt{D}+m > m(\floor*{k\sqrt{D}}+1) = \sum a_i^2+b_i^2D \geq \frac{m^2k^2}{4t}+tD,\] 
for $D\in I_t\left(\frac{mk}{2}\right)$. This is impossible if $\sqrt{D} \geq \frac{mk}{2t}+\frac{\sqrt{m}}{\sqrt{t}}$ or if $\sqrt{D} \leq \frac{mk}{2t}-\frac{\sqrt{m}}{\sqrt{t}}$. Given that $D\in I_t\left(\frac{mk}{2}\right)$, this means that if \[ \sqrt{D} \in  \left [\frac{mk}{2\sqrt{t(t+1)}}, \frac{mk}{2t}-\frac{\sqrt{m}}{\sqrt{t}}\right]\cup \left [\frac{mk}{2t}+\frac{\sqrt{m}}{\sqrt{t}}, \frac{mk}{2\sqrt{(t-1)t}}\right], \]
$m\alpha$ can not be represented as the sum of squares.

Combining parts of these constraints for $t$ and $t-1$, we get the interval from the theorem statement: \[\sqrt{D} \in\left [\frac{mk}{2t}+\frac{\sqrt{m}}{\sqrt{t}}, \frac{mk}{2(t-1)}-\frac{\sqrt{m}}{\sqrt{t-1}}\right].\] For this interval to be non-empty, inequalities \[\frac{mk}{2t}+\frac{\sqrt{m}}{\sqrt{t}} \leq \frac{mk}{2\sqrt{(t-1)t}}\leq \frac{mk}{2(t-1)}-\frac{\sqrt{m}}{\sqrt{t-1}}\] 
must hold, which happens for $m\geq \frac{4t\left(t-1\right)\left(2t-1+2\sqrt{t(t-1)}\right)}{k^2}$.
The above interval is not well-defined for $t=1$, for which we immediately get the right bound to be $\infty$ just from considering a single inequality for $D\in I_1\left(\frac{mk}{2}\right)$.

\textbf{(b)} Now let's look at the case $D\equiv 1 \pmod{4}$.

This time $\alpha = \floor*{\frac{k\sqrt{D}-k}{2}}+1+\frac{k+k\sqrt{D}}{2}$ and again $m\alpha = \sum_i (\alpha_i)^2$. However, this time $\alpha_i = \frac{a_i}{2}+\frac{b_i\sqrt{D}}{2}$ with $a_i\equiv b_i\pmod{2}$. We can again assume $a_i\geq0$ and $b_i\geq 0$ since $\frac{1+D}{4}+\frac{\sqrt{D}}{2}>m$ for $D\geq 4m$. By comparing irrational parts, we get $\sum a_ib_i = mk$. Analogously, comparing rational parts and using Lemma \ref{lemma:parity}, for arbitrary $t\equiv m \pmod{2}$ we get
\[\frac{mk\sqrt{D}}{2}+m > m\left(\floor*{\frac{k\sqrt{D}-k}{2}}+1+\frac{k}{2}\right) = \sum \frac{a_i^2+b_i^2D}{4} \geq \frac{m^2k^2}{4t}+\frac{tD}{4}\]
for $D \in J_t(mk)$. The rest of the proof is identical, except this time we obtain intervals
\[\left[\frac{mk}{t}+\frac{2\sqrt{m}}{\sqrt{t}}, \frac{mk}{t-2}-\frac{2\sqrt{m}}{\sqrt{t-2}}\right],\]
which are non-empty for $m\geq \frac{t(t-2)(2t-2+2\sqrt{t(t-2)}}{k^2}$. Again, cases $t=1$ and $t=2$ have to be considered separately.
\end{proof}

%This figure should be somewhere around the paragraphs below
\begin{figure}[h]
\includegraphics[width=12cm]{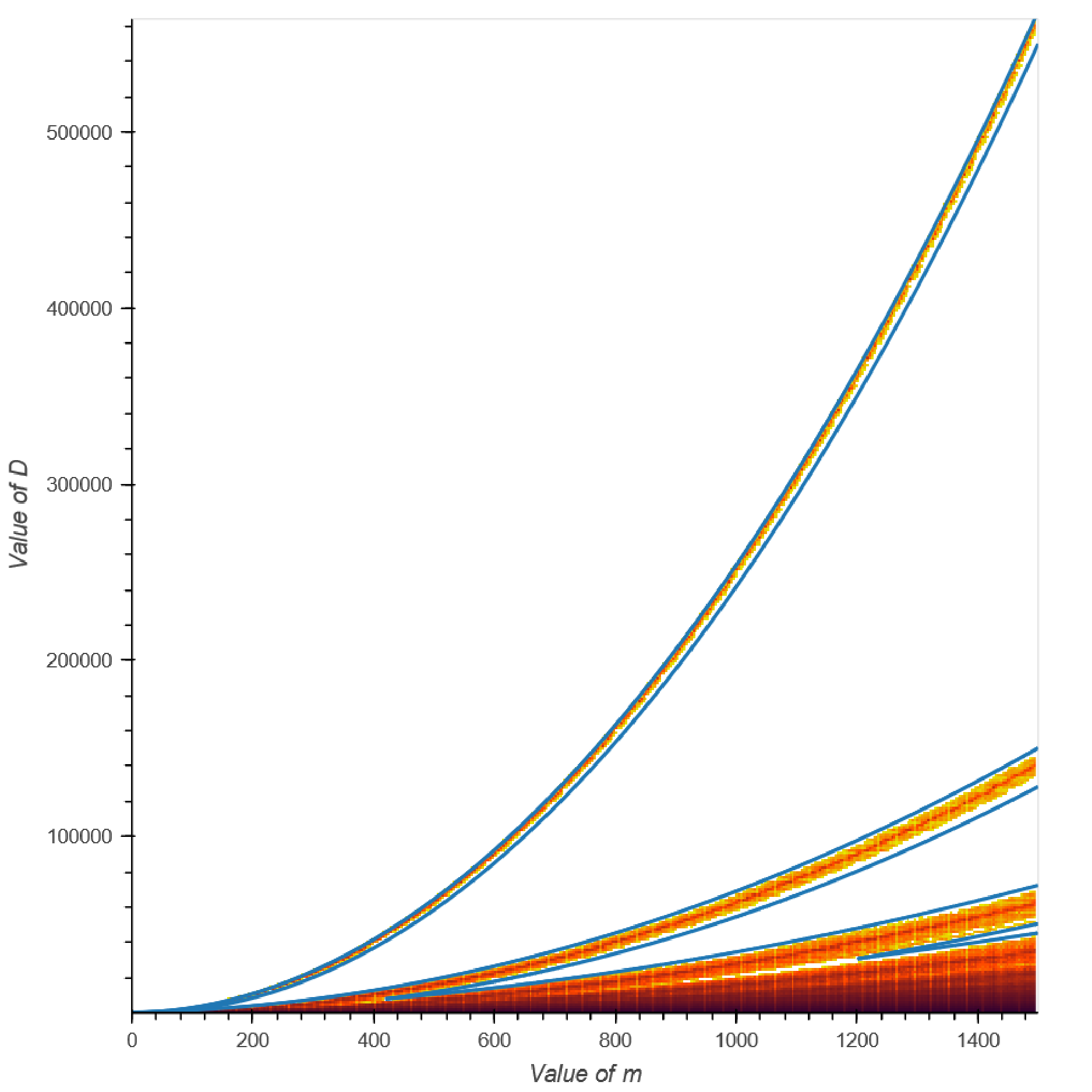}
\centering
\caption{Case $D\equiv 2,3 \pmod{4}$. Red dots represent pairs $(m,D)$ such that all elements of $m\kO^+$ are sums of squares.}
\label{fig:examplegraph}
\end{figure}

The proof of Proposition \ref{proposition:main} has a similar structure as the one used by \cite{sosil} to prove that if $D>4m^2$ for $D\equiv 2,3 \pmod{4}$, resp. $D>16m^2$ for $D\equiv 1 \pmod{4}$, then not all elements of $m\kO^+$ are sums of squares. Proposition \ref{proposition:main} shows that the boundary is actually around $\frac{m^2}{4}$, resp $m^2$ (simply consider the proposition for $k=1$ and $t$ the lowest possible).

The benefit of Proposition \ref{proposition:main} is that not only it gives better bounds, it also gives quite good insight into the structure of the problem and dependence of $D$ on $m$. For example, if $D\equiv 2,3 \pmod{4}$, these solutions can be only ``clustered'' around values $m^2/4$, $m^2/16$, $\ldots$, $m^2/4i^2$ and with increasing $m$, we get more accurate approximations for how large these clusters can be. The basic restrictions can be seen just by simply considering $k=1$. The interesting question is what is the intersection of all these intervals for given $m$. As we will show in the next theorem, for parameters $k$, $t$ with ratio $k:t$ close to a fixed value, these intervals can be typically grouped together to create one large interval between $m/2i$ and $m/2(i-1)$. This behavior can be seen in Figure \ref{fig:examplegraph}, containing computed data as well as these summarized bounds obtained by Theorem \ref{theorem:main}. 

\begin{customthm}{2}
Let $K=\Q(\sqrt{D})$ with $D\geq 2$ squarefree and a positive integer $m$. If $\sqrt{D}$ lies in one of the following intervals, then \emph{not} all elements of $m\kO^+$ can be represented as the sum of squares:
\begin{enumerate}
\item $\left [ \frac{m}{2}+4 ,\infty\right)$,
$ \left[\frac{m}{2i}+iC_1,  \frac{m}{2(i-1)}-(i-1)C_2 \right ]$ for integer $i>1$ and $D \equiv 2,3 \pmod{4}$,
\item $\left [ \frac{m}{2}+8 ,\infty\right)$,
$ \left[ \frac{m}{2i}+2iC_1,  \frac{m}{2(i-1)}-2(i-1)C_2 \right ]$ for integer $i>1$, $D \equiv 1 \pmod{4}$ and even $m$,
\item $\left [ m+4 ,\infty\right)$, $ \left[ \frac{m}{2i+1}+(4i+2)C_1,  \frac{m}{2i-1}-(4i-2)C_2 \right ]$ for integer $i>0$, $D \equiv 1 \pmod{4}$ and odd $m$,
\end{enumerate}
where constants $C_1$, $C_2$ are defined as $C_1 = \sqrt{24+16\sqrt{2}}$ and $C_2 = \sqrt{48+24\sqrt{3}}$.
\end{customthm}
\begin{proof}
We will prove only the case (a), where $D \equiv 2,3 \pmod{4}$. The other cases can be handled analogously.

By AM-GM inequality $\frac{m}{2i}+iC_1 \geq \sqrt{2mC_1}\geq \sqrt{m}$, hence we can assume $\sqrt{D} \geq \sqrt{m}$ and it does not restrict the desired intervals. By Proposition \ref{proposition:main}, we have the intervals
\[S(t,k) = \left [\frac{mk}{2t}+\frac{\sqrt{m}}{\sqrt{t}}, \frac{mk}{2(t-1)}-\frac{\sqrt{m}}{\sqrt{t-1}}\right]\]
 such that if $\sqrt{D}$ lies in one of them, then \emph{not} all elements of $m\kO^+$ can be represented as the sum of squares.
For a fixed $i$, we want to look at the union of all these intervals $S(t,k)$ intersected with the interval $\left [\frac{m}{2i}, \frac{m}{2(i-1)} \right]$. 

At first, let's consider the case $i=1$ where we claim $\left [ \frac{m}{2} ,\infty\right) \cap \bigcup_{t,k\geq 1} S(t,k) \supset \left [ \frac{m}{2}+4 ,\infty\right)$. In fact, we will show that it is sufficient to consider only the union $\bigcup_{t\geq 1} S(t,t)$. The lower bound of the intervals $S(t,t)$ gets lower with increasing $t$, so if we want their union to be an interval, we need to check that the upper bound of $S(t+1,t+1)$ is larger than the lower bound of $S(t,t)$. This is equivalent to the inequality \[\frac{mt}{2t}+\frac{\sqrt{m}}{\sqrt{t}} \leq \frac{m(t+1)}{2t}-\frac{\sqrt{m}}{\sqrt{t}},\]
which is by a trivial computation true for $m\geq 16t$. Since we consider a fixed value of $m$ and $t$ is integer we get a bound $t\leq \floor{\frac{m}{16}}$ for which we can unite the intervals $S(t,t)$ to get an interval.  When $m\geq 16t$ holds, the intervals $S(t,t)$ are also all non-empty (the bounds from Proposition \ref{proposition:main} are satisfied). So in total
\[ \bigcup_{t=1}^{\floor{\frac{m}{16}}+1} S(t,t) = \left [ \frac{m}{2}+\sqrt{\frac{m}{\floor{\frac{m}{16}}+1}} ,\infty\right) \supset \left [ \frac{m}{2}+4 ,\infty\right). \]

In the case $i>1$, we will show that  \[ \bigcup_{k\geq 1} S(ki,k) \cup  \bigcup_{k\geq 1} S(k(i-1)+1,k) \supset \left[ \frac{m}{2i}+iC_1,   \frac{m}{2(i-1)}-(i-1)C_2  \right ], \] which will prove the theorem. At first, observe that for $k=1$, it holds $S(k(i-1)+1,k) = S(ki,k) = S(i,1)$. Now, let's look at $\bigcup_{k\geq 1} S(ki,k)$, where $S(ki,k) = \left[ \frac{m}{2i}+\frac{\sqrt{m}}{\sqrt{ki}},   \frac{mk}{2(ki-1)}-\frac{\sqrt{m}}{\sqrt{ki-1}}  \right ]$. Clearly with increasing $k$, the lower bound of $S(ki,k)$ decreases. Hence if we prove that $S(ki,k) \cap S((k+1)i,k+1) \neq \emptyset$ for all $1 \leq k < K$, where $K$ is some positive integer, then $\bigcup_{k\geq 1} S(ki,k) \supset \left[ \frac{m}{2i}+\frac{\sqrt{m}}{\sqrt{Ki}},   \frac{m}{2(i-1)}-\frac{\sqrt{m}}{\sqrt{i-1}}  \right ]$. With this knowledge, the condition $S(ki,k) \cap S((k+1)i,k+1) \neq \emptyset$ can be rewritten as

\[ \frac{m}{2i}+\frac{\sqrt{m}}{\sqrt{ki}} \leq \frac{m(k+1)}{2((k+1)i-1)}-\frac{\sqrt{m}}{\sqrt{(k+1)i-1}}, \]

which is by a straightforward computation equivalent to 

\begin{equation}
\label{prf-bound1}
   m \geq 4i((k+1)i-1) \frac{2ki+i-1+2\sqrt{ki}\sqrt{(k+1)i-1}}{k} . 
\end{equation}

Just to verify, if (\ref{prf-bound1}) is satisfied then the inequality

\[ m\geq \frac{4(k+1)i\left((k+1)i-1\right)\left(2(k+1)i-1+2\sqrt{(k+1)i((k+1)i-1)}\right)}{(k+1)^2} \]

from Proposition \ref{proposition:main} for $S((k+1)i,k+1)$ to be non-empty is satisfied as well.

Denote $K>0$ the smallest positive integer such that (\ref{prf-bound1}) is not satisfied for $k=K$. Then
\[ m < 4i((K+1)i-1) \frac{2Ki+i-1+2\sqrt{Ki}\sqrt{(K+1)i-1}}{K} < 4i^3K\left(1+\frac{1}{K}\right)\left(2+\frac{1}{K}+2\sqrt{1+\frac{1}{K}}\right),\]
which using $K\geq 1$ yields $m < i^3K(24+16\sqrt{2})$.
Finally, by using this inequality, $\frac{m}{2i}+\frac{\sqrt{m}}{\sqrt{Ki}} < \frac{m}{2i}+i\sqrt{24+16\sqrt{2}}$, therefore \[ \bigcup_{k\geq 1} S(ki,k) \supset \left[ \frac{m}{2i}+\frac{\sqrt{m}}{\sqrt{Ki}},   \frac{m}{2(i-1)}-\frac{\sqrt{m}}{\sqrt{i-1}}  \right ] \supset \left[ \frac{m}{2i}+i\sqrt{24+16\sqrt{2}},   \frac{m}{2(i-1)}-\frac{\sqrt{m}}{\sqrt{i-1}}  \right ]. \]

Now let's look at the union $\bigcup_{k\geq 1} S(k(i-1)+1,k)$, where \[S(k(i-1)+1,k) = \left[ \frac{mk}{2(k(i-1)+1)}+\frac{\sqrt{m}}{\sqrt{(k(i-1)+1)}},   \frac{m}{2(i-1)}-\frac{\sqrt{m}}{\sqrt{k(i-1)}}  \right ].\] This time clearly the upper bounds are increasing with increasing $k$, so analogously the condition $S(k(i-1)+1,k) \cap S((k+1)(i-1)+1,k+1) \neq \emptyset$ translates to the inequality

\[ \frac{m}{2(i-1)}-\frac{\sqrt{m}}{\sqrt{k(i-1)}} \geq \frac{m(k+1)}{2((k+1)(i-1)+1)}+\frac{\sqrt{m}}{\sqrt{((k+1)(i-1)+1)}}, \]

which is by a straightforward computation equivalent to 

\begin{equation}
\label{prf-bound2}
   m \geq 4(i-1)((k+1)(i-1)+1) \frac{2k(i-1)+i+2\sqrt{k(i-1)}\sqrt{(k+1)(i-1)+1}}{k} . 
\end{equation}
Again, if (\ref{prf-bound2}) is satisfied then the inequalities from Proposition \ref{proposition:main} for $S((k+1)(i-1)+1,k+1)$ and $S(k(i-1)+1,k)$  to be non-empty are satisfied as well.

Analogously, denote $K>0$ the smallest positive integer such that (\ref{prf-bound2}) is not satisfied for $k=K$. Then
\[ m < 4(i-1)^3K\left(1+\frac{1}{K}+\frac{1}{K(i-1)}\right)\left(2+\frac{i}{(i-1)K}+2\sqrt{1+\frac{1}{K}+\frac{1}{K(i-1)}}\right),\]
which using $K\geq 1$ and $i \geq 2$ yields $m < (i-1)^3K(48+24\sqrt{3})$.
Finally, by using this inequality, $\frac{m}{2(i-1)}-\frac{\sqrt{m}}{\sqrt{K(i-1)}} > \frac{m}{2(i-1)}-(i-1)\sqrt{48+24\sqrt{3}}$, therefore \[ \bigcup_{k\geq 1} S(k(i-1)+1,k) \supset \left[ \frac{m}{2i}+\frac{\sqrt{m}}{\sqrt{i}},   \frac{m}{2(i-1)}-(i-1)\sqrt{48+24\sqrt{3}}  \right ]. \]

\end{proof}

Again, it is important to note that for a fixed $m$, only finitely many of these intervals are non-empty.
Constants $C_1$ and $C_2$ are chosen so that the theorem can be stated for all $m$ simultaneously. As can be seen in the proof, for $m \rightarrow \infty$ the variable $K \rightarrow \infty$ and the optimal constants $C_1,C_2 \rightarrow 4$.

To conclude this section, let's summarize all the known necessary and sufficient bounds for $D$ in terms of $m$.

\begin{cor}\label{cor:bounds}
Let $K = \Q(\sqrt{D})$ with $D\geq 2$ squarefree.
\begin{enumerate}
    \item If $D\equiv 2,3 \pmod{4}$ and $D\geq \left(\frac{m}{2}+4\right)^2$, then \emph{not} all elements of $m\kO^+$ are represented as the sum of squares in $\kO$.
    \item If $D\equiv 1 \pmod{4}$, $m$ is even and $D\geq \left(\frac{m}{2}+8\right)^2$, then \emph{not} all elements of $m\kO^+$ are represented as the sum of squares in $\kO$.
    \item If $D\equiv 1 \pmod{4}$, $m$ is odd and $D\geq \left(m+4\right)^2$, then \emph{not} all elements of $m\kO^+$ are represented as the sum of squares in $\kO$.
    \item If $D\leq m$ for $D\equiv 2,3 \pmod{4}$ and even $m$ or if $D\leq 2m$ for $D\equiv 1 \pmod{4}$, all elements of $m\kO^+$ are represented as the sum of squares in $\kO$.
    \item If $m$ is odd and $D\equiv 2,3 \pmod{4}$, then \emph{not} all elements of $m\kO^+$ are represented as the sum of squares in $\kO$.
\end{enumerate}
\end{cor}
\begin{proof}
Parts $(a)-(c)$ are direct consequences of Theorem \ref{theorem:main}.
Parts $(d), (e)$ has already been proven by \cite{sosil} in Theorem \ref{cor:2}. Proof of the part $(d)$ will be outlined in the next section, whilst
the part $(e)$ is a consequence of the comparison of irrational parts in the sum of squares, which can be seen in the proof of Proposition \ref{proposition:main}.
\end{proof}

In Section \ref{sec:indecom}, we will look closely at some specific values of $D$ to show that these restrictions are in some sense accurate.

\section{Peters theorem} \label{sec:peters}

In this section, we will try to look at the other side of the problem, which is showing that for chosen $D$ and $m$ all elements of $m\kO^+$ are represented as the sum of squares. The key element of the following work will be a result proved by Peters \cite{P}, which characterizes when an element is the sum of squares.

\begin{theorem}{\textup{\cite[Satz 2]{P}}} \label{theorem:peters}
Let $K = \Q(\sqrt{D})$ with $D>0$ squarefree. Then $\alpha \in \kO^+$ of the form 
\[
\alpha=
\left\{
	\begin{array}{ll}
		a+b\frac{1+\sqrt{D}}{2} & \mbox{if } D\equiv 1 \pmod{4}, \\
		  a+2b\sqrt{D} &\mbox{if } D\equiv 2,3 \pmod{4},
	\end{array}
\right. \]
is the sum of $5$ squares if and only if there exist rational integer $c$, with additional condition $c\equiv b \pmod{2}$ if $D\equiv 1 \pmod{4}$, such that 
\[
c\in
\left\{
	\begin{array}{ll}
		\left[\frac{2a+b-2\sqrt{N(\alpha)}}{D},\frac{2a+b+2\sqrt{N(\alpha)}}{D} \right]& \mbox{if } D\equiv 1 \pmod{4}, \vspace{0.5em}\\
		\left[ \frac{a-\sqrt{N(\alpha)}}{2D},\frac{a+\sqrt{N(\alpha)}}{2D} \right] &\mbox{if } D\equiv 2,3 \pmod{4}.
	\end{array}
\right.
\]
\end{theorem}

One direct consequence of this theorem is Corollary $\ref{cor:bounds}$ (d). The norm of every element in $m\kO^+$ is at least $m^2$; therefore, if $m$ is large enough, the interval has the length of at least $1$, resp. $2$, so it must contain the desired integer.

Using the acquired necessary bounds and this theorem, we have a method for determining all $D$ such that all elements of $m\kO^+$ are represented as the sum of $5$ squares. For example, for $m=4$ we get

\begin{theorem} \label{theorem:m=4}
Let $K=\Q(\sqrt{D})$ with $D\geq2$ squarefree. Then every element of $4\kO^{+}$ is the sum of squares in $\kO$ if and only if $D \in \{2,3,5,6,7,10,11,13\}$.
\end{theorem}
\begin{proof}
``$\Rightarrow$'' Assume all elements of  $4\kO^{+}$ can be represented as the sum of squares in $\kO$.

At first, consider $D\equiv 2,3 \pmod{4}$. Using Proposition \ref{proposition:main} for $t=1$, $k=1$, we get the bound $D < 16$. Furthermore, it can be easily verified that for $D \in \{14,15\}$, the element $4(\floor{\sqrt{D}}+1+\sqrt{D})$ does not satisfy the condition in Theorem $\ref{theorem:peters}$.

If $D\equiv 1 \pmod{4}$, we obtain the bound $D \leq 23$ similarly. However, the values $D\in\{17,21\}$ can be again excluded using the element $4\left(\floor*{\frac{1+\sqrt{D}}{2}} + \frac{1+\sqrt{D}}{2}\right)$.

``$\Leftarrow$'' Consider $D\equiv 2,3 \pmod{4}$ and first, let's look at the general idea of the proof. Consider any element $\alpha = a+b\sqrt{D} \in \kO^{+}$. Using Theorem \ref{theorem:peters}, $4\alpha$ is the sum of squares if and only if there exists a rational integer in the interval

\[\left[\frac{2a-2\sqrt{N(\alpha)}}{D},\frac{2a+2\sqrt{N(\alpha)}}{D}\right].\]
If $4\sqrt{N(\alpha)}\geq D$, this is clearly true. It is, therefore, enough to consider only the elements of small norm (which can be, in general, done by solving finitely many 
generalized Pell equations).
\begin{enumerate}
    \item $D=2,3$. In this case, $4\sqrt{N(\alpha)}\geq D$ is always true.
    \item $D=6$. In this case, the interval $\left[\frac{a-\sqrt{N(\alpha)}}{3},\frac{a+\sqrt{N(\alpha)}}{3}\right]$ always contains an integer since $N(\alpha)\geq1$.
    \item $D=7$. It can be seen that the only potentially problematic case is $a\equiv \pm 2 \pmod{7}$ (otherwise, $\frac{2a}{7}$ is at most $\frac 2 7$ away from the nearest integer). However, in that case, we get $N(\alpha)\equiv a^2 \equiv 4 \pmod{D}$ and $N(\alpha)\geq 4$ is indeed enough.
    \item $D=10$. In this case, only $a\not\equiv\pm1\pmod{5}$ and $N(\alpha)\leq 3$ would cause a problem. Analogously, $a\not\equiv\pm1\pmod{5}$ implies $a^2\equiv N(\alpha) \equiv -1 \pmod{5}$, which results in $N(\alpha)\geq 4$.
    \item $D=11$. In this case $a\equiv \pm 2 \pmod{11}$ causes a problem for $N(\alpha)\leq 3$, $a\equiv \pm 3$  for $N(\alpha)\leq 6$, $a\equiv \pm 4$  for $N(\alpha)\leq 2$. All these cases can be dealt with in the same fashion as above.
\end{enumerate}
Consider $D \equiv 1 \pmod{4}$, $\alpha = a+b\omega_D \in \kO^{+}$. By Theorem \ref{theorem:peters}, we need to prove there is an even integer in the interval 
\[\left[\frac{4(2a+b)-4\sqrt{4N(\alpha)}}{D},\frac{4(2a+b)+4\sqrt{4N(\alpha)}}{D}\right],\]
which is equivalent to the existence of an integer in the interval
\[\left[\frac{2(2a+b)-2\sqrt{4N(\alpha)}}{D},\frac{2(2a+b)+2\sqrt{4N(\alpha)}}{D}\right],\]
For $D=5$, this is clear. The only remaining case is $D=13$.

Here, $4N(\alpha) = 4(a+\frac b 2)^2-4(\frac{b\sqrt{D}}{2})^2 = (2a+b)^2-b^2D\equiv (2a+b)^2 \pmod{D}$.

The only problematic residues are:
\begin{enumerate}
    \item $(2a+b)\equiv \pm 3\pmod{13}$ and $4N(\alpha)\leq 8$ -- not possible using the congruence above, 
    \item $(2a+b)\equiv \pm 4\pmod{13}$ and $4N(\alpha)\leq 6$ -- not possible using the congruence above and the fact that $N(\alpha)$ is an integer.
\end{enumerate}
\end{proof}
In the second part of the proof, we dealt with the bad cases by clever manipulations with congruences and inequalities. This option might not be viable in the general case, but we can always solve corresponding Pell equations to get the elements of the given norm and get the desired congruences using their well-known structure and recurrence relations. On the other hand, if $D$ does not meet the conditions, we will inevitably find a counterexample during this process.

Using this technique, one can (dis)prove the given statement for arbitrarily chosen $m\kO^{+}$ and $D$. However, it gets progressively more tedious with increasing $D$ and heavily relies on the ability to quickly solve the generalized Pell equation. In the next section, we will introduce an improved algorithm that uses indecomposable elements of $\kO^+$ to avoid this issue.

\section{The indecomposables}\label{sec:indecom}

The characterization of indecomposable elements of $\kO^+$ in quadratic fields can be used to show a few concrete results and also a general algorithm for solving the given problem.

The main idea is that instead of considering all elements of $\kO^+$, we can just look at the indecomposables since if $a$ and $b$ can be both represented as the sum of squares, so can $a+b$. For $D$ of a specific form, the indecomposables have such a nice structure that it can be used in combination with Theorem \ref{theorem:peters} to fully characterize which $m$ satisfy the given statement. One easy consequence of the following theorems is that bounds in Corollary \ref{cor:bounds} are, in some sense, optimal.

\begin{theorem}\label{theorem:t2-1}
Let $K=\Q(\sqrt{D})$ with squarefree $D=t^2-1$ for some even integer $t>1$. For a positive rational integer $m$, the following are equivalent:
\begin{enumerate}
    \item All elements of $2m\kO^{+}$ are represented as the sum of squares in $\kO$.
    \item $m=(t-1)k+l$ for some $k\geq0$ and $0\leq l \leq 2k$.
\end{enumerate}

\end{theorem}
\begin{proof}
In all cases, $D\equiv 3 \pmod{4}$. 
If we prove that for all indecomposable elements $\alpha \in \kO^{+}$, $2m\alpha$ can be represented as sum of squares in $\kO$, then obviously all elements of $2m\kO^{+}$ can be represented.

The continued fraction representation of $\sqrt{t^2-1}$ is $[t-1,\overline{1,2(t-1)}]$. For $i\geq -1$, define a sequence $(\alpha_i)$ by $\alpha_{-1} = 1$, $\alpha_{0} = t-1 +\sqrt{D}$, $\alpha_{i+2} = \alpha_{i+1}+\alpha_i$ for odd $i\geq -1$ and $\alpha_{i+2} = 2(t-1)\alpha_{i+1}+\alpha_i$ for even $i\geq 0$.

Using the facts mentioned in Section \ref{sec:prelim}, all indecomposable elements in $\kO^{+}$ are in this case  $\alpha_i$ with odd $i\geq -1$, together with their conjugates. If $\alpha_i$ can be represented as the sum of squares, then its conjugate can be represented as well (just replace all the squared numbers by their conjugates), so we can consider only $\alpha_i$.

It can be easily proven by induction that $\alpha_{2k+1} = \alpha_1^{k+1} = (t+\sqrt{t^2-1})^{k+1}$ either using the recurrence relations above or using facts from Section \ref{sec:prelim} ($\alpha_1 = \varepsilon$, $\alpha_{i+2} = \varepsilon \alpha_i$)  

The important consequence it that $2m\alpha_{2k+1}$ can be represented as the sum of squares for all $k\geq -1$ if and only if $2m\alpha_1$ can be represented as the sum of squares. One implication is trivial and the other one comes from the following facts. If $k$ is odd, then $\alpha_{2k+1} = \left(\alpha_1^{\frac{k+1}{2}}\right)^2$, so its is already a square in $\kO$. If $k$ is even $2m\alpha_{2k+1} = 2m\alpha_1 \cdot \alpha_{2k-1}$, both factors can be represented as the sum of squares in $\kO$ so the product can be represented as well.

By Theorem \ref{theorem:peters}, $2m\alpha_1 = 2m(t+\sqrt{t^2-1})$ can be represented as the sum of squares if and only if there is a rational integer in the interval
\[\left[\frac{2mt-\sqrt{N(2m\alpha_1)}}{2D},\frac{2mt+\sqrt{N(2m\alpha_1)}}{2D}\right] = \left[\frac{m}{t+1},\frac{m}{t-1}\right].\]
By substituing $m=(t-1)k+l$ for some integers $k\geq 0$, $0\leq l < t-1$, it can be seen that this is true if and only if $l=0$ or $\frac{m}{t+1}\leq k$, which is equivalent to $l\leq 2k$.
\end{proof}

The biggest $m$ not satisfying the condition $(b)$ is $m = (t-1)\frac{t-4}{2}+t-2 = \frac{t^2-3t}{2}$, meaning there exist infinitely many pairs $(D,m) = (t^2-1,\frac{t^2-3t}{2})$ such that not all elements of $2m\kO^+$ are represented as the sum of squares. This shows the sufficient bound $D \leq \frac{m}{2}$ in Corollary \ref{cor:bounds}(d) is quite precise. 

On the other hand, the smallest non-zero $m$ satisfying the condition is $m = t-1$, therefore all elements of $2m\kO^+$ are represented as the sum of squares for $D = (m+1)^2+1$ (if this $D$ is squarefree to be precise). This, again, shows that the necessary bound $D<(m+4)^2$ from Corollary \ref{cor:bounds}(a) is somewhat accurate.

\begin{theorem}\label{theorem:2t-4}
Let $K=\Q(\sqrt{D})$ with squarefree $D=(2t+1)^2-4$ for some integer $t>1$. For positive rational integer $m$, all elements of $m\kO^{+}$ are represented as the sum of squares in $\kO$ if and only if $m$ is one of the following form:
\begin{enumerate}
    \item $m=(4t-2)k+l$ for some $k\geq0$ and $0\leq l \leq 8k$ for $l$ even,
    \item  $m=(4t-2)k+l$ for some $k\geq0$ and $0\leq l \leq 8k-2t-3$ for $l$ odd,
    \item $m=(4t-2)k+l$ for some $k\geq0$ and  $2t-1\leq l \leq 8k+2t+3$ for $l$ odd. 
\end{enumerate}
\end{theorem}
\begin{proof}
The proof is quite similar to the proof of the previous theorem. This time, $D\equiv 1 \pmod{4}$ and $\omega_D = \frac{1+\sqrt{D}}{2} = [t,\overline{1,2t-1}]$ and again it holds that indecomposable elements are exactly powers of fundamental unit $ \alpha_1 = t+\omega'_D = t+\frac{-1+\sqrt{D}}{2}$.

So it is equivalent to looking at when $m\alpha_1$ is the sum of squares, which, using Theorem \ref{theorem:peters}, is exactly if and only if there is an integer $n\equiv m \pmod{2}$ in the interval 
\[\left[\frac{2mt+m-2m}{D},\frac{2mt+m+2m}{D}\right] = \left[\frac{m}{2t+3},\frac{m}{2t-1}\right].\]
Let $m = k(4t-2)+l$ for $k\geq 0$ and $4t-2 > l\geq 0$. Then 
\[\left[\frac{m}{2t+3},\frac{m}{2t-1}\right] = \left[\frac{m}{2t+3},2k+\frac{l}{2t-1}\right].\]
If $l$ is even, so is $m$, and the condition is satisfied if and only if $\frac{m}{2t+3}\leq 2k$, which is equivalent to $l\leq 8k$. 

If $l$ is odd and $l<2t-1$, one needs $\frac{m}{2t+3}\leq 2k-1$, equivalently $l \leq 8k-2t-3$. If $l\geq 2t-1$, only $\frac{m}{2t+3}\leq 2k+1$ is needed, which is equivalent to $l \leq 8k+2t+3$.

One can see that $m = k(4t-2)+l$ with the discussed restrictions satisfy the interval condition even if we omit $l<4t-2$. This concludes the proof.
\end{proof}

The biggest even $m$ and the biggest odd $m$ not satisfying the conditions are both in the set $\{2t^2-t-2$, $2t^2-3t-1\}$. This, again, shows that the sufficient bound $D \leq 2m$ in $\ref{cor:bounds}(d)$ is quite sharp.

Regarding the smallest even $m$ satisfying the conditions, we get $m=4t-2$, while the smallest odd $m$ is $m = 2t-1$. These are in close agreement with  the necessary condition $D < (\frac{m}{2}+8)^2$, resp. $D < (m+4)^2$, from Corollary $\ref{cor:bounds}(b,c)$.

A similar approach could be naturally extended to other quadratic families, resulting in a system of (in)equalities that characterize all suitable $m$.

\section{Algorithmic solution}\label{sec:alg}

The methods used in the previous section can be generalized to construct an algorithm that determines if all elements of $m\kO^+$ are represented as the sum of squares. As was mentioned, one only needs to consider $m$-multiples of the indecomposables. If $s$ is the period of the continued fraction of $D$, then $\varepsilon = \alpha_{s-1}$ is the fundamental unit and $\alpha_{i+s} = \varepsilon\alpha_i$. Therefore, there are only finitely many indecomposable elements (resp. their $m$-multiples) up to conjugation and multiplication by $\varepsilon^2$. If and only if all of them are represented, so are all elements of $m\kO^+$. And we can easily check each of the elements using Theorem \ref{theorem:peters}. 

Unique indecomposable elements in $\kO^+$ (uniqueness in the sense of the previous paragraph) are exactly $\alpha_{i,r}$ with odd $-1 \leq i \leq 2s-3$ and $0\leq r \leq u_{i+2}-1$. Therefore, the number of unique indecomposables is $\sum_{i=1}^s u_{2i-1}$. Kala and Blomer \cite[Theorem 2]{indBK} showed this sum to be $O(\sqrt{D}(\log(D))^2)$, which is therefore also the time complexity of this algorithm for single $m$ and $D$. 

Let's look at the time complexity for determining all $D$ satisfying the conditions for fixed $m$. At first, consider only $D \equiv 2,3 \pmod{4}$ and assume $D$ satisfies the conditions. Necessary bounds in Theorem $\ref{theorem:main}$ imply $D < \left(\frac{m}{2}+4\right)^2$ and also $D \not\in \left[\left ( \frac{m}{2i}+iC_1\right)^2, \left ( \frac{m}{2(i-1)}-(i-1)C_2\right)^2 \right]$. We consider this interval only for finitely many $i$ satisfying $ \frac{m}{2i}+iC_1 < \frac{m}{2(i-1)}-(i-1)C_2$ -- the largest one being $i_{max} \asymp m^{\frac{1}{3}}$ (here notation $f \asymp g$ means $C |g(x)| < |f(x)| < D |g(x)|$ for some positive constants $C,D$ and all sufficiently large $x$). Therefore, the smallest $D$ excluded by the intervals is $D_{min} \asymp m^{\frac{4}{3}}$. As a consequece, if $D$ satisfies the conditions, then either $D < D_{min}$ or $D$ lies in one of the intervals  $\left[\left(\frac{m}{2i}-iC_2\right)^2,\left(\frac{m}{2i}+iC_1\right)^2 \right]$ for $i < i_{max}$. In the first case we can estimate the time complexity as $\sum_{D<D_{min}} O(\sqrt{D}(\log(D))^2) = O(m^{\frac{4}{3}} \cdot m^{\frac{4}{6}}(\log(m^{\frac{4}{3}}))^2) = O(m^2(\log(m))^2)$. In the second case, we have $O(m^{\frac{1}{3}})$ intervals and the length of each interval is $O(m)$. Therefore, the estimation of time complexity in the second case is  $\sum_{i<i_{max}} O(m)\left(\frac{m}{2i} + iC_1\right)(2\log(\frac{m}{2i} + iC_1))^2 = O(m^2(\log(m))^2) \left( \sum_{i<i_{max}}\frac{1}{i}\right) + O(m(\log(m))^2)\left( \sum_{i<i_{max}}i \right) = O(m^2(\log(m))^3)$. In the last equality we used $\sum_{i=1}^n \frac{1}{i} = O(\log(n))$ (see e.g. \cite[Theorem 3.2]{apo}).

Overall, the time complexity for $D\equiv 2,3 \pmod{4}$ is $O(m^2(\log(m))^3)$ and analogous argument can be made for $D\equiv 1 \pmod{4}$. Hence, the time complexity of this algorithm (that for fixed $m$ determines for which $D$ all elements of $m\kO^+$ can be represented as the sum of squares) is also $O(m^2(\log(m))^3)$.

The table below shows results for a few small $m$. The implementation in C++ (available at \url{https://github.com/raskama/number-theory/tree/main/quadratic}) was used to obtain results for $m\leq 5000$.
These data, as well as generated graphs, can be found at
\url{https://www2.karlin.mff.cuni.cz/\~raskam/research/quad/}.

\begin{table}[H]
\centering
\begin{tabular}{|c||c|} 
 \hline
 $m$ & $D$ such that all elements of $m\kO^+$ are sums of squares \\ 
 \hline\hline
 1 & 5 \\ \hline
 2 & 2, 3, 5 \\ \hline
 3 & 5, 13, 17, 21 \\ \hline
 4 & 2, 3, 5, 6, 7, 10, 11, 13 \\ \hline
 5 & 5, 13, 17, 21, 29, 37  \\ \hline
6 & 2, 3, 5, 6, 7, 10, 11, 13, 14, 15, 17, 21, 26, 29, 33 \\ \hline
7 & 5, 13, 17, 21, 29, 33, 37, 41, 53, 61, 65, 77 \\ \hline
8 & 2, 3, 5, 6, 7, 10, 11, 13, 14, 15, 17, 19, 21, 22, 23, 26, 29, 31, 37, 38, 53 \\ \hline
9 & 5, 13, 17, 21, 29, 33, 37, 41, 53, 57, 61, 65, 69, 77, 85, 93, 101 \\ \hline
10 & \begin{tabular}{c} 2, 3, 5, 6, 7, 10, 11, 13, 14, 15, 17, 19, 21, 22, \\23, 26, 29, 30, 33, 34, 35, 37, 38, 41, 43, 53, 65, 85\end{tabular}
  \\ \hline
11 & 5, 13, 17, 21, 29, 33, 37, 53, 57, 65, 73, 77, 85, 101, 145, 165\\ \hline

\end{tabular}

\label{table:1}
\end{table}

\section{Acknowledgments}
I would like to thank V\'it\v{e}zslav Kala very much for introducing me to this engaging topic and for overall guidance and patience on this project.


\begin{thebibliography}{99}

\bibitem[Apo]{apo}
{\scshape Apostol, T. M.} (1976).  \emph{Introduction to Analytic Number Theory}  New York, NY: Springer-Verlag.

\bibitem[BK]{indBK}
{\scshape Blomer, V., Kala, V.} (2018).  On the rank of universal quadratic forms over real quadratic fields. \emph{Doc. Math.} 23: 15–34. \DOI{10.25537/dm.2018v23.15-34}

\bibitem[CP]{CP} {\scshape Cohn, H., Pall, G.} (1962). Sums of four squares in a quadratic ring. \emph{Trans. Amer. Math. Soc.} 105: 536–556. \DOI{10.1090/S0002-9947-1962-0142522-8}

\bibitem[DS]{Sindecomp} 
{\scshape Dress, A., Scharlau, R.} (1982). Indecomposable totally positive numbers in real quadratic orders. \emph{J. Number Theory}. 14(3): 292-306. \DOI{10.1016/0022-314X(82)90064-6}

\bibitem[HK]{indecompos} 
{\scshape Hejda, T., Kala, V.} (2020).  Additive structure of totally positive quadratic integers. \emph{Manuscripta Math}. 163: 263-278. \DOI{10.1007/s00229-019-01143-8}


\bibitem[KRS]{biquad} 
{\scshape Kr\'asensk\'y, J.,
Ra\v{s}ka, M., Sgallov\'a, E.} (2022). Pythagoras numbers of orders in biquadratic fields. \emph{Expo. Math.} \DOI{10.1016/j.exmath.2022.06.002}
%ORIGINAL:
%\bibitem[KRS]{biquad} 
%{\scshape Kr\'asensk\'y, J.,
%Ra\v{s}ka, M., Sgallov\'a, E.} Pythagoras numbers of orders in biquadratic fields. Preprint: \href{https://arxiv.org/abs/2105.08860}{arXiv:2105.08860}.


\bibitem[KT]{cubic} 
{\scshape Kala, V., Tinkov\'a, M.} (2022). Universal quadratic forms, small normes and traces over families of number fields. \emph{Int. Math. Res. Not. IMRN.} \DOI{10.1093/imrn/rnac073}
%ORIGINAL:
%\bibitem[KT]{cubic} 
%{\scshape Kala, V., Tinkov\'a, M.}  Universal quadratic forms, small normes and traces over families of number fields. Preprint: \href{https://arxiv.org/abs/2005.12312}{arXiv:2005.12312}.


\bibitem[KY]{sosil}
{\scshape Kala, V., Yatsyna, P.} (2020). 
 Sums of squares in $S$-integers. \emph{New York J. Math.} 26: 1145-1154.

\bibitem[Ma]{mass} {\scshape Maa\ss, H.} (1941).  \"Uber die Darstellung total positiver Zahlen des K\"orpers $R(\sqrt{5})$ als Summe von drei Quadraten. \emph{Abh.
Math. Sem. Univ. Hamburg}. 14: 185-191. \DOI{10.1007/BF02940744}


\bibitem[Pe]{per} 
{\scshape Perron O.} (1913).  \emph{Die Lehre von den Kettenbr{\"u}chen, 1st edn}. Leipzig: B. G. Teubner Verlag.

\bibitem[Pet]{P} {\scshape Peters,  M.} (1973).  Quadratische {F}ormen \"{u}ber {Z}ahlringen. \emph{Acta Arith}. 24: 157-164.  \DOI{10.4064/aa-24-2-157-164}

\bibitem[Si1]{siegel} {\scshape Siegel, C. L.} (1921). Darstellung total positiver Zahlen durch Quadrate. \emph{Math. Z.} 11: 246-275. \DOI{10.1007/BF01203627}

\bibitem[Si2]{Si} {\scshape Siegel, C. L. } (1945).  Sums of m-th powers of algebraic integers. \emph{Ann. of Math.} 46: 313-339. \DOI{10.2307/1969026}

\bibitem[Ti]{tin} 
{\scshape  Tinkov\'a, M. } On the Pythagoras number of the simplest cubic fields. Preprint: \href{https://arxiv.org/abs/2101.11384}{ 	arXiv:2101.11384}.

\end{thebibliography}
\end{document}